\documentclass[a4paper,10pt]{amsart}

\usepackage{a4wide}
\usepackage[latin1]{inputenc} 
\usepackage[T1]{fontenc}

\usepackage{amsfonts}
\usepackage{amsmath,amssymb,amsthm,stmaryrd}

\usepackage{mathrsfs}
\usepackage{esint}

\usepackage{fancyhdr}
\usepackage{verbatim}

\usepackage{graphicx}
\usepackage{ifpdf}
\ifpdf
\fi 

\DeclareMathOperator{\dist}{dist}
\DeclareMathOperator{\diam}{diam}

\newcommand{\ud}[0]{\,\mathrm{d}}

\newcommand{\Babs}[1]{\Big|#1\Big|}

\newcommand{\R}{\mathbb{R}}

\newcommand{\Z}{\mathbb{Z}}
\newcommand{\N}{\mathbb{N}}

\swapnumbers
\numberwithin{equation}{section}

\makeatletter
  \let\c@equation\c@subsection
\makeatother

\theoremstyle{plain}
\newtheorem{theorem}[subsection]{Theorem}
\newtheorem{lemma}[subsection]{Lemma}

\newtheorem{proposition}[subsection]{Proposition}

\theoremstyle{definition}
\newtheorem{definition}[subsection]{Definition}

\newtheorem{examples}[subsection]{Examples}

\theoremstyle{remark}
\newtheorem{remark}[subsection]{Remark}

\makeatletter
\@namedef{subjclassname@2010}{%
  \textup{2010} Mathematics Subject Classification}
\makeatother

\title{What is a cube?}

\author{Tuomas Hyt\"onen}
\author{Anna Kairema}

\address{Department of Mathematics and Statistics, P.O.B. 68 (Gustaf H\"allstr\"omin katu 2), FI-00014 University of Helsinki, Finland}
\email{tuomas.hytonen@helsinki.fi, anna.kairema@helsinki.fi}

\thanks{T.H. is supported by the European Research Council's Starting Grant ``Analytic-probabilistic methods for borderline singular integrals''.
Both authors are supported by the Academy of Finland, grant 133264.}

\begin{document}

\begin{abstract}
We give an intrinsic characterization of all subsets of a doubling metric space that can arise as a member of some system of dyadic cubes on the underlying space, as constructed by M.~Christ.
\end{abstract}

\maketitle

\section{Introduction}

The notion of a cube in the usual Euclidean space does not need much explanation. \emph{Dyadic cubes} are then certain special cubes with particular coordinate representations. The indispensable role of the dyadic cubes in Harmonic Analysis on Euclidean spaces has also motivated the construction of analogous structures in more general settings, most notably by M.~Christ \cite{Christ90} in doubling metric spaces. However, this leads to a slight change in the point-of-view: there is no longer a notion of a `cube' as such, and even a `dyadic cube' barely makes sense as an individual object; it only becomes meaningful as a member of a \emph{system of dyadic cubes} with useful intersection and covering properties reminiscent of those in the Euclidean case. Nevertheless, 
%
it is natural to ask the following question, which was posed to one of us by F.~Bernicot:
\begin{quote}
  What assumptions on a set do I have to put such that it can be considered one of
the dyadic sets of a suitable dyadic system? \cite{Bernicot:personal}
\end{quote}

In this note, we give a complete answer to this question, provided that a `suitable dyadic system' is understood in the sense of the construction by Christ, which seems to be the most useful one at least for problems of singular integrals, and which we recall below. But let us first discuss the motivation to understand Bernicot's question.

First, many common arguments in Euclidean Harmonic Analysis involve the dyadic subcubes of a given (a priori, non-dyadic) cube. While it is quite clear what this means in the Euclidean space, the notion of a `dyadic subcube' seems to become meaningless in an abstract space, unless we started from a dyadic cube from the beginning. Our characterization, however, provides an explicit way of testing whether a set qualifies as a dyadic cube. For such a set $E$, the existing techniques may be further pushed to yield a dyadic system $\mathscr{D}$ with $E\in\mathscr{D}$. After this, the dyadic subcubes of $E$ come as a part of the construction.

Another situation is the following: After the seminal work of Nazarov--Treil--Volberg \cite{NTV}, it is now standard to treat singular integrals with respect to a non-doubling measure on $\R^n$ with the help of a random choice of the system of dyadic cubes. Since any cube of $\R^n$ can arise as a random dyadic cube in their construction, it is necessary to impose certain assumptions, such as the `accretivity'
\begin{equation*}
  \Babs{\frac{1}{\mu(Q)}\int_Q b\ud\mu}\geq\delta>0
\end{equation*}
on the testing function $b$ in the $Tb$ theorem, over the family of all cubes $Q\subset\R^n$. The $Tb$ theorem of Nazarov--Treil--Volberg was generalized to the setting of an abstract metric space $X$ by Hyt\"onen and Martikainen \cite{HM09}, but there it was left unclear, for which sets $Q\subset X$ exactly it is necessary to impose the above accretivity condition. The present characterization of all sets that can arise as dyadic cubes gives a clean form of this condition in the mentioned theorem.


The set-up for our characterization is the following. Let $(X,d)$ be a metric space. We assume that $X$ has the following \emph{(geometric) doubling property}: There exists a positive integer $A_1\in \N$ such that for every $x\in X$ and $r>0$, the ball $B(x,r):=\{y\in X:d(y,x)<r\}$ can be covered by at most $A_1$ balls $B(x_i,r/2)$.

To state our characterization, we formulate the following notion, which goes back to Martio--V\"ais\"al\"a \cite{MartioVaisala1} (a similar condition was used in \cite{MartioVuorinen}):

\begin{definition}
A set $E\subseteq X$ is \textit{plump} with parameters $R>0$ and $b\in (0,1)$ if:
\begin{equation}\label{eq:property}
\text{For all $y\in E$ and $0<r\leq R$, there exists $z\in X$ such that $B(z,br)\subseteq B(y,r)\cap E$.}
\end{equation}
\end{definition}

It turns out that a set $E$ can arise as a dyadic cube in $X$ if and only if both $E$ and $X\setminus E$ are plump, more precisely:

\begin{theorem}\label{thm:main}
Let $(X,d)$ be a geometrically doubling metric space.
Given $E\subseteq X$, the Christ-type dyadic cubes may be constructed in such a way that $\tilde{Q}\subseteq E\subseteq \bar{Q}$, where $\tilde{Q}$ and $\bar{Q}$ are the interior and closure of some dyadic cube $Q$, if and only if $E$ is bounded and both $E$ and $X\setminus E$ are plump with parameters $b\in (0,1)$ (depending only on the space) and $R\gtrsim \diam E$.
\end{theorem}

We will provide a more precise quantitative formulation of the result in Proposition~\ref{prop:cube_is_plum} and Proposition~\ref{prop:existsDyadSystem}. The properties of dyadic cubes will be recalled in Section \ref{sec:dyadic_cubes}.

\section{Definitions and lemmas}

We begin this section by recalling the dyadic structures. After that, we recall and study the notion of plumpness defined in the Introduction.

\subsection{Dyadic cubes}\label{sec:dyadic_cubes}
In a geometrically doubling metric space $(X,d)$, a family of Borel sets $Q^k_\alpha, k\in \Z, \alpha\in I(k)$, is called \textit{a system of dyadic cubes with parameters} $\delta\in (0,1)$ and $0<c_1\leq C_1<\infty$ if it has the following properties:
\begin{equation}\label{eq:cover}
  X=\bigcup_{\alpha\in I(k)}Q^k_{\alpha}\quad\text{(disjoint union)}\quad\forall k\in\Z;
\end{equation}
\begin{equation}\label{eq:nested}
   \text{if }\ell\geq k\text{, then either }Q^{\ell}_{\beta}\subseteq Q^k_{\alpha}\text{ or }Q^k_{\alpha}\cap Q^{\ell}_{\beta}=\emptyset;
\end{equation}
\begin{equation}\label{eq:contain}
  B(x^k_{\alpha},c_1\delta^k)\subseteq Q^k_{\alpha}\subseteq B(x^k_{\alpha},C_1\delta^k)=:B(Q^k_{\alpha});
\end{equation}
\begin{equation}\label{eq:monotone}
   \text{if }\ell\geq k\text{ and }Q^{\ell}_{\beta}\subseteq Q^k_{\alpha}\text{, then }B(Q^{\ell}_{\beta})\subseteq B(Q^k_{\alpha}).
\end{equation}
The set $Q^k_\alpha$ is called a \textit{dyadic cube of generation} $k$ centered at $x^k_\alpha$ with side length $\delta^k$. The interior and closure of $Q^k_\alpha$ are denoted by $\tilde{Q}^k_{\alpha}$ and $\bar{Q}^k_{\alpha}$, respectively. It follows from the geometric doubling property that the index set $I(k)$ is at most countable for each value of $k\in \Z$, and it can be assumed to be an initial interval in $\N$.

\begin{definition}
We say that a set $\{x^k_\alpha\}_{k,\alpha}\subseteq X$ is a \emph{system of dyadic points} with parameters $\delta\in(0,1)$ and $0<c_0\leq C_0<\infty$ if the following properties hold for every $k\in\Z$:
\begin{equation}\label{ehdot}
  d(x_{\alpha}^k,x_{\beta}^k)\geq c_0\delta^k\quad(\alpha\neq\beta),\qquad
  \min_{\alpha}d(x,x^k_{\alpha})< C_0\delta^k\quad\forall\,x\in X.
\end{equation}

We say that  a partial order $\leq$ among the index pairs $(k,\alpha)$ is a \emph{dyadic partial order} for a given system of dyadic points, if it satisfies the following properties:
\begin{itemize}
  \item Every $(k+1,\beta)$ satisfies $(k+1,\beta)\leq(k,\alpha)$ for exactly one value of $\alpha$.
  \item For $\ell\leq k$, we have $(\ell,\beta)\leq(k,\alpha)$ if and only if $\ell=k$ and $\beta=\alpha$.
  \item For $\ell>k$, we have $(\ell,\beta)\leq(k,\alpha)$ if and only if there exist $\eta_k=\alpha,\eta_{k+1},\ldots,\eta_{\ell-1},\eta_{\ell}=\beta$ such that $(j+1,\eta_{j+1})\leq(j,\eta_j)$ for every $j\in\{k,k+1,\ldots,\ell-1\}$.
  \item The relation $(k+1,\beta)\leq(k,\alpha)$ is almost determined by the proximity of the points in the sense of the two implications
  \begin{equation*}
  d(x^{k+1}_\beta,x^k_\alpha)<\frac12 c_0\delta^k\ \Rightarrow\ 
  (k+1,\beta)\leq(k,\alpha)\ \Rightarrow\  d(x^{k+1}_\beta,x^k_\alpha)<C_0\delta^k.
\end{equation*}
\end{itemize}

\end{definition}

We recall from \cite{oma} the following result, which is a slight elaboration of seminal work by M. Christ \cite{Christ90}:

\begin{theorem}\label{thm:existence}
Let $(X,d)$ be a geometrically doubling metric space. Suppose that there is a system of dyadic points $\{x^k_\alpha\}_{k,\alpha}$ with parameters $\delta\in(0,1)$ and $0<c_0\leq C_0<\infty$ that satisfy $12C_0\delta\leq c_0$,
 and a dyadic partial order $\leq$ among the index pairs $(k,\alpha)$. Then there exists a system of dyadic cubes $Q^k_\alpha$  with parameters $\delta$, $c_1=\tfrac13 c_0$, $C_1=2C_0$, and centre points $x^k_\alpha$. In fact, this system can be constructed in such a way that
\begin{equation}\label{eq:3cubes}
  \tilde{Q}^k_\alpha\subseteq Q^k_\alpha\subseteq \bar{Q}^k_\alpha,
\end{equation}
where
\begin{equation}\label{eq:closedCube}
  \bar{Q}^k_\alpha=\overline{\{x^{\ell}_\beta:(\ell,\beta)\leq(k,\alpha)\}}
\end{equation}
and
\begin{equation}\label{eq:openCube}
  \tilde{Q}^k_\alpha=\operatorname{int}\bar{Q}^k_\alpha=
  \Big(\bigcup_{\gamma\neq\alpha}\bar{Q}^k_\gamma\Big)^c.
\end{equation}
\end{theorem}

If either the system of points or the partial order is not given a priori, their existence already follows from the assumptions; however, we want to emphasize the point that any given system of points and partial order can be used as a starting point.

\begin{remark}\label{rem:inclusion}
The proof \cite{oma} shows that the second inclusion in \eqref{eq:contain} is true with $Q^k_\alpha$ replaced by $\bar{Q}^k_\alpha$. 
\end{remark}


\subsection{Plumpness}

We recall from the Introduction that a set $E\subseteq X$ is said to be \textit{plump with parameters} $R>0$ and $b\in (0,1)$ if $E$ satisfies the following:
\begin{equation}\label{eq:PlumpProperty}
\text{For all $y\in E$ and $0<r\leq R$, there exists $z\in X\colon B(z,br)\subseteq B(y,r)\cap E$.}
\end{equation}

\begin{remark}
`Plumpness' has a close connection with other geometric notations. It is easily verified that in $\R^n$, plumpness is equivalent to the \textbf{corkscrew condition} by Jerison and Kenig \cite{JerisonKenig}: A domain $\Omega$ in $\R^n$ satisfies the interior (exterior) corkscrew condition if for some $R >0$ and $b\in (0,1)$, and every $y\in\partial \Omega$ and $0< r< R$, there exists a non-tangential point $z\in B(y,r)\cap\Omega$ ($z\in B(y,r)\cap(\R^n\setminus\bar{\Omega})$) such that $\dist(z,\partial\Omega)\geq br$. 
\end{remark}

\begin{examples}
1. Examples of plump sets in $\R^n$ are provided by \textbf{John domains}, first introduced by F. John \cite{John}: A domain $\Omega$ in $\R^n$ is $(\alpha,\beta)$-John domain if there exists a point $x_0\in \Omega$ (`central point') such that given any $x\in \Omega$, there exists a rectifiable path $\gamma\colon [0,\ell]\to\Omega$ which is parametrized by arclength, such that $\gamma(0)=x, \gamma(\ell)=x_0, \ell\leq\beta$ and
\[\dist(\gamma(t),\partial\Omega)\geq \frac{\alpha}{\ell}\, t\quad \forall\; t\in[0,\ell]. \]  
Every John domain satisfies the corkscrew condition \cite[Lemma 6.3]{MartioVuorinen} and thus, is a plump set. 

2. The well established \textbf{non-tangentially accessible domains} (NTA domains), introduced by Jerison and Kenig \cite{JerisonKenig}, satisfy both the interior and exterior corkscrew condition by definition. Thus, by the main result of the present paper, every bounded NTA domain in $\R^n$ qualifies as a dyadic cube.

\end{examples}

We record the following easy observation.

\begin{lemma}\label{lem:plump}
Suppose $E$ is plump with parameters $R>0$ and $b\in (0,1)$. Then 

(1) $E$ is plump with any parameters $\tilde{R}\leq R$ and $0<\tilde{b}\leq b$;

(2) $E$ is plump with any parameters $\tilde{R}\geq R$ and $\tilde{b}\in (0,1)$ that satisfy $\tilde{R}\tilde{b}\leq Rb$.
\end{lemma}
\begin{proof}
The first assertion is obvious. For the second assertion, let $\tilde{R}\geq R$ and $\tilde{b}\in (0,1)$ be such that $\tilde{R}\tilde{b}\leq Rb$. Suppose $y\in E$ and $0<r\leq \tilde{R}$. Then $0<rR/\tilde{R}=:t\leq R$ so that there exists $z\in E$ such that
\begin{equation*}
 B(z,\tilde{b}r)=B(z,\frac{\tilde{b}\tilde{R}}{R}t)\subseteq B(z,bt)\subseteq B(y,t)\cap E
  \subseteq B(y,r)\cap E.\qedhere
\end{equation*}
\end{proof}

We give next a definition for a dyadically plump set, that better suits our purposes. 

\begin{definition}
A set $E\subseteq X$ is \textit{dyadically plump (d-plump) with parameters} $\delta\in (0,1), m\in \Z$ and $0<b_0 \leq B_0<\infty$ if $E$ satisfies the following:
\begin{equation}\label{eq:dyadicPlumpProperty}
\text{For all $y\in E$ and $k\geq m$, there exists $z\in X\colon B(z,b_0\delta^k)\subseteq B(y,B_0\delta^k)\cap E$.}
\end{equation}
\end{definition}

Qualitatively, set $E$ is plump if and only if $E$ is d-plump. Quantitatively, the relationship is formulated in the following lemma.

\begin{lemma}\label{lem:plumpness}
If $E$ is a plump set with parameters $R>0$ and $b\in (0,1)$, then $E$ is d-plump with any parameters $\delta\in (0,1), m\in \Z$ and $0<b_0 \leq B_0<\infty$ that satisfy $b_0/B_0\leq b$ and $B_0\delta^m\leq R$. Conversely, if $E$ is a d-plump set with parameters $\delta\in (0,1), m\in \Z$ and $0<b_0 \leq B_0<\infty$, then $E$ is plump with any parameters $R>0$ and $b\in (0,1)$ that satisfy $b\leq \delta b_0/B_0$ and $R\leq B_0\delta^{m-1}$.
\end{lemma}

\begin{proof}
For the first assertion, let $\delta\in (0,1)$, and suppose that $0<b_0\leq B_0<\infty$ are such that $b_0/B_0\leq b$. Then pick $m\in \Z$ that satisfies $B_0\delta^m\leq R$. Let $y\in E$ and $k\geq m$. Then, by \eqref{eq:PlumpProperty} with $r=B_0\delta^k\leq B_0\delta^m\leq R$, there exists $z\in X$ such that
\[B(z,b_0\delta^k)\subseteq B(z,br)\subseteq B(y,r)\cap E=B(y,B_0\delta^k)\cap E,\]
which shows that $E$ is d-plump.

For the second assertion, suppose that $b>0$ and $R>0$ are such that $b\leq \delta b_0/B_0$ and $R\leq B_0\delta^{m-1}$. Let $y\in E$ and $0<r\leq R$, and let $k\geq m$ be an integer that satisfies $B_0\delta^{k}<r\leq B_0\delta^{k-1}$. Then, by \eqref{eq:dyadicPlumpProperty}, there exists $z\in X$ such that
\begin{equation*}
  B(z,br)\subseteq B(z,b_0\delta^k)\subseteq B(y,B_0\delta^k)\cap E \subseteq B(y,r)\cap E. \qedhere
\end{equation*}
\end{proof}

\section{The proof of the main result}

In this section we will provide a proof for the quantitative version of our main result, Theorem~\ref{thm:main}, which is formulated in Propositions~\ref{prop:cube_is_plum} and \ref{prop:existsDyadSystem} below. 

\begin{proposition}\label{prop:cube_is_plum}
Suppose that $\mathscr{D}$ is a family of dyadic cubes with parameters $\delta\in (0,1)$ and $0<c_1\leq C_1<\infty$. Let $Q^m_\alpha\in\mathscr{D}$ be a dyadic cube of side length $\delta^m$. Then both $F\in \{Q^m_\alpha , X\setminus Q^m_\alpha \}$ are d-plump with parameters $\delta ,m, b_0=c_1$ and $B_0=c_1+C_1$. In particular, $Q^m_\alpha$ is plump with parameters $b=\delta c_1/(c_1+C_1)$ and
\begin{equation*}
  R= \frac{c_1+C_1}{2\delta C_1}\diam(Q^m_\alpha).
\end{equation*}
\end{proposition}

\begin{proof}
It suffices to consider the case $F=Q^m_\alpha$ and $y\in Q^m_\alpha$ since otherwise, $y\in Q^m_\beta$ for some $\beta\neq \alpha$, and we argue similarly with $\alpha$ replaced by $\beta$.

Suppose $y\in Q^m_\alpha$ and $k\geq m$. By \eqref{eq:cover} and \eqref{eq:nested}, $y\in Q^{k}_\beta$ for some $Q^{k}_\beta\subseteq Q^{m}_\alpha$. Thus, by \eqref{eq:contain},
\begin{equation}\label{eq:2}
  d(y,x^{k}_\beta )<C_1\delta^{k}.
\end{equation}
We will show that $z=x^k_\beta$ satisfies \eqref{eq:dyadicPlumpProperty}. First note that $B(x^k_\beta,c_1\delta^k)
\subseteq Q^{k}_\beta\subseteq Q^m_\alpha$.
We are left to show that $B(x^{k}_\beta,c_1\delta^k) \subseteq B(y,B_0\delta^k)$. To this end, suppose $x\in B(x^{k}_\beta,c_1\delta^k)$ and note that, by \eqref{eq:2},
\begin{equation*}
  d(x,y)\leq d(x,x^{k}_\beta)+d(x^{k}_\beta,y)
  \leq c_1\delta^k + C_1\delta^{k}
  =B_0\delta^k.
\end{equation*}
This shows that $E$ is d-plump with parameters $\delta, m, b_0=c_1$ and $B_0=c_1+C_1$.
By Lemma~\ref{lem:plumpness}, $Q^m_\alpha$ is plump and we may choose $b=\delta b_0/B_0=\delta c_1/(c_1+C_1)$ and $R=B_0\delta^{m-1}=(c_1+C_1)\delta^{m-1}$. Since $\diam(Q^m_\alpha)\leq 2C_1\delta^m$ by \eqref{eq:contain}, the proof is completed by Lemma~\ref{lem:plump}(1).
\end{proof}

\begin{proposition}\label{prop:existsDyadSystem}
Let $E\subseteq X$, and suppose both $F\in \{E, X\setminus E\}$ are d-plump with parameters $\delta\in (0,1), m\in\Z$ and $0<b_0\leq B_0<\infty$ where $\diam E\leq B_0\delta^m$ and $12B_0\delta\leq b_0$. Then the Christ-type dyadic cubes may be constructed in such a way that $E$ arises as a dyadic cube. More precisely, there exists a dyadic system $\mathscr{D}$ with parameters $\delta, c_1=b_0/3$ and $C_1=2B_0$, and a dyadic cube $Q\in\mathscr{D}$ of side length $\delta^m$ such that $\tilde{Q}\subseteq E\subseteq \bar{Q}$. 
\end{proposition}

The proof of Proposition~\ref{prop:existsDyadSystem} consists of three lemmata.

\begin{lemma}[Choice of dyadic points]\label{lem:dyadic_points}
Under the assumptions and with the fixed values of parameters as in Proposition~\ref{prop:existsDyadSystem}, let $F\in \{E, X\setminus E\}$. Then for every $k\in \Z$, there exists a set $\{x^k_\alpha\}$ of points with the following properties: For all $k\in X$,
\begin{equation}\label{eq:distribution1}
  d(x^k_\alpha,x^k_\beta)\geq b_0\delta^k \quad (\alpha\neq \beta),\qquad
  \min_\alpha d(x,x^k_\alpha)<B_0\delta^k\quad\forall\ x\in X;
\end{equation}
If $k\geq m$, then moreover
\begin{equation}\label{eq:distribution2}
  \min_{\alpha:x^k_\alpha\in F} d(x,x^k_\alpha)<B_0\delta^k\quad\forall\ x\in F\in\{E,X\setminus E\};
\end{equation}
and
\begin{equation}\label{eq:far_from_bdr}
  \dist(x^k_\alpha ,X\setminus F)\geq b_0\delta^k \quad \forall\ x^k_\alpha\in F.
\end{equation}
If $k=m$, then there is exactly one $\alpha$ such that $x^m_\alpha\in E$.
\end{lemma}

\begin{proof}
We first observe that
\begin{equation}\label{eq:nonEmpty}
  \{x\in F:d(x,X\setminus F)\geq b_0\delta^m\}\neq\emptyset
\end{equation}
for both choices of $F$. To this end, pick a $y\in F$. We apply \eqref{eq:dyadicPlumpProperty} with $k=m$, and find a point $z\in X$ such that
\begin{equation*}
  B(z,b_0\delta^m)\subseteq F\cap B(y,B_0\delta^m),
\end{equation*}
and thus $d(z,X\setminus F)\geq b_0\delta^m$.

For $k\geq m$ and both choices of $F$, we pick a maximal set $\{x^k_\alpha\}_\alpha$, of points in $F$ that satisfies the two conditions
\begin{equation*}
d(x^k_\alpha,x^k_\beta)\geq b_0\delta^k\quad (\alpha\neq\beta)\quad
\text{and} \quad \dist(x^k_\alpha ,X\setminus F)\geq b_0\delta^k .
\end{equation*}
By \eqref{eq:nonEmpty}, both these collections are nonempty. We equip them with individual labels to form one joint collection $\{x^k_\alpha\}_\alpha$. It still satisfies
\begin{equation*}
  d(x^k_\alpha,x^k_\beta)\geq b_0\delta^k\quad(\alpha\neq\beta):
\end{equation*}
if both $x^k_\alpha,x^k_\beta$ belong to the same $F\in\{E,X\setminus E\}$, this is part of the construction, and if $x^k_\alpha\in F$, $x^k_\beta\in X\setminus F$, then
\begin{equation*}
  d(x^k_\alpha,x^k_\beta)\geq d(x^k_\alpha,X\setminus F)\geq b_0\delta^k\quad(\alpha\neq\beta).
\end{equation*} 
For $k=m$, we through away all but one $x^m_\alpha\in E$, which we denote by $x^m_{\alpha_0}$.

We need to check that the points $x^k_\alpha$ are $B_0\delta^k$-dense in $F$ for both choices of $F$. If $F=E$ and $k=m$, then all $x\in E$ satisfy $d(x,x^m_{\alpha_0})\leq\diam E<B_0\delta^m$. Let then either $k>m$ or $k=m$ and $F=X\setminus E$, and consider an arbitrary $x\in F$. We apply \eqref{eq:dyadicPlumpProperty}, and find a point $z\in F$ such that $B(z,b_0\delta^m)\subseteq F\cap B(x,B_0\delta^m)$. In particular, $\dist (z,X\setminus F)\geq b_0\delta^m$. If also $d(z,x^m_\alpha)\geq b_0\delta^m$ for all $x^m_\alpha\in F$, then $z$ could have been added to the collection $\{x^m_\alpha\}_{\alpha}$, contradicting its maximality. Thus, there exists $x^m_\alpha\in B(z,b_0\delta^m)\subseteq F\cap B(x,B_0\delta^m)$; i.e., $x^m_\alpha\in F$ and $d(x,x^m_\alpha)<B_0\delta^m$, which is as claimed in \eqref{eq:distribution2}.

Finally, for $k<m$, we pick a maximal set $\{x^k_\alpha\}_\alpha$, of points in $X$ that satisfy the first condition in \eqref{eq:distribution1}, and then by maximality also the second condition in \eqref{eq:distribution1}, since $B_0\geq b_0$. For $k>m$, there are no further conditions required, so we are done.
\end{proof}

Note that the point set $\{x^k_\alpha\}_{k,\alpha}$ provided by Lemma~\ref{lem:dyadic_points} is, in particular, a system of dyadic points with parameters $\delta$ and $0<b_0\leq B_0<\infty$ that satisfy $12B_0\delta \leq b_0$.

The next step in the construction of dyadic cubes is the choice of a partial order for the dyadic index pairs $(k,\alpha)$, which describes the child-parent (descendant-ancestor) relationships.

\begin{lemma}[Choice of the dyadic partial order]\label{lem:partial_order}
Under the assumptions and with the fixed values of parameters as in Proposition~\ref{prop:existsDyadSystem}, there is a dyadic partial order $\leq$ among the pairs $(k,\alpha)$ with the following additional property:
\begin{quote}
  If $k\geq m$ and $(\ell,\beta)\leq(k,\alpha)$, then $x^\ell_\beta$ and $x^k_\alpha$ belong to the same set $F\in\{E,X\setminus E\}$.
\end{quote}
\end{lemma}

\begin{proof}
We define a partial order as follows. 
Given $k\geq m$ and a point $x^{k+1}_\beta\in F$, check whether there exists $\alpha$ such that $d(x^k_\alpha,x^{k+1}_\beta)<b_0\delta^k/2$. If one exists, it is necessarily unique by \eqref{eq:distribution1}, and moreover, $x^{k}_\alpha\in F$ by \eqref{eq:far_from_bdr}. We then decree that $(k+1,\beta)\leq (k,\alpha)$. If no such good $\alpha$ exists, choose any $\alpha$ for which $x^k_\alpha\in F$ and $d(x^k_\alpha,x^{k+1}_\beta)<B_0\delta^k$, and decree that $(k+1,\beta)\leq (k,\alpha)$; at least one such $\alpha$ exists by \eqref{eq:distribution2}. In either case, we decree that $(k+1,\beta)$ is not a child of any other $(k,\gamma)$. Note that the additional above property is clear from this construction.

Given $k<m$ and a point $(k+1,\beta)$, we proceed in the same way as before except that we drop the requirement $x^k_\alpha\in F$. Finally, we extend $\leq$ by transitivity to obtain a partial ordering.
\end{proof}

With the dyadic points and the partial order at hand, Theorem~\ref{thm:existence} guarantees the existence of a system of dyadic cubes $Q^k_\alpha$ with parameters $\delta$, $c_1=\tfrac13 b_0$ and $C_1=2B_0$. The proof of Proposition~\ref{prop:existsDyadSystem} is now completed by the following lemma.


\begin{lemma}
If $\alpha_0$ is the unique index with $x^m_{\alpha_0}\in E$, we have
$\tilde{Q}^m_{\alpha_0} \subseteq E\subseteq \bar{Q}^m_{\alpha_0}$.
\end{lemma}
\begin{proof}
Suppose $x\in E$. Then, by \eqref{eq:distribution2}, for every $k\geq m$ there exists $x^k_\beta\in E$ such that $d(x^k_\beta ,x)<B_0\delta^k\to 0$ as $k\to\infty$. This shows that
\[x\in \overline{\{x^k_\beta\colon (k,\beta)\leq (m,\alpha_0)\}}=\bar{Q}^m_{\alpha_0}.\]
To show that $\tilde{Q}^m_{\alpha_0} \subseteq E$, it suffices to show that 
\[X\setminus E\subseteq \left(\tilde{Q}^m_{\alpha_0}\right)^c=\bigcup_{\alpha\neq\alpha_0}\bar{Q}^m_\alpha .\]
To this end, suppose $x\in X\setminus E$. Then, by \eqref{eq:distribution2}, for every $k\geq m$ there exists $x^k_\beta\in X\setminus E$ such that $d(x^k_\beta ,x)<B_0\delta^k\to 0$ as $k\to\infty$. By Lemma~\ref{lem:partial_order}, for each such $k$ we have that $(k,\beta)\leq (m,\alpha)$ with some $x^m_\alpha\in X\setminus E$. This implies that $x^k_\beta\in \bar{Q}^m_\alpha$ for some $\alpha\neq\alpha_0$, and hence $d(x^k_\beta,x^m_\alpha)<2B_0\delta^m$ by \eqref{eq:contain}; see also Remark~\ref{rem:inclusion}. For all such $\alpha$, we have $d(x^m_\alpha,x)\leq d(x^m_\alpha,x^k_\beta)+d(x^k_\beta,x)<2B_0\delta^m+B_0\delta^k\leq 3B_0\delta^m$, and hence geometric doubling implies that there are only boundedly many relevant $x^m_\alpha$ here. Passing to a subsequence over $k$, we may assume that all $x^k_\beta$ belong to the same $\bar{Q}^m_{\alpha_1}$ with $\alpha_1\neq\alpha_0$. Thus also
\begin{equation*}
  x=\lim_{k\to\infty} x^k_\beta\in\bar{Q}^m_{\alpha_1}\subseteq\bigcup_{\alpha\neq\alpha_0}\bar{Q}^m_{\alpha_0}.\qedhere
\end{equation*}
\end{proof}

This completes the proof of Proposition~\ref{prop:existsDyadSystem}, and thereby also the proof of Theorem~\ref{thm:main}.

\bibliographystyle{plain}

\def\cprime{$'$}

\end{document}